\documentclass[11pt]{amsart}

\usepackage{amsmath,amsthm,amssymb,latexsym,amstext,amsfonts}
\usepackage[english]{babel}
\usepackage{geometry}
\usepackage{graphicx}

\usepackage{mathdots}
\usepackage{braket}
\usepackage{tikz}
\usetikzlibrary{snakes}
\usetikzlibrary{arrows}

\usepackage{hyperref} 

\newtheorem{theorem}{Theorem}[section]

\newtheorem{corollary}[theorem]{Corollary}

\newtheorem{definition}[theorem]{Definition}

\newtheorem{lemma}[theorem]{Lemma}

\newtheorem{proposition}[theorem]{Proposition}
\newtheorem{remark}[theorem]{Remark}

\def \RR{{\mathcal R}}

\def \LL{{\mathcal L}}
\def \DD{{\mathcal D}}

\newcommand{\Drel}{\mathbin{\mathcal D}}
\newcommand{\Rrel}{\mathbin{\mathcal R}}
\newcommand{\Lrel}{\mathbin{\mathcal L}}

\title{\emph{FLAT COSET DECOMPOSITIONS OF SKEW LATTICES}}

\author[Jo\~ao Pita Costa and Karin Cvetko-Vah]{
Jo\~ao Pita Costa* and Karin Cvetko-Vah** \bigskip \\
* Institut Jo\v zef  Stefan,\\
Jamova Cesta 39, 1000 Ljubljana, Slovenia.\\
joao.pitacosta@ijs.si\bigskip \\
** University of Ljubljana,\\
Faculty of Mathematics and Physics,\\
Jadranska 19, 1000 Ljubljana, Slovenia.\\
karin.cvetko@fmf.uni-lj.si\\
tel. +386 1 476 66 24 fax. 386 1 251 72 81
}

\date{\today}

\begin{document}

\maketitle

\begin{abstract}
Skew lattices are non-commutative generalizations of lattices, and the cosets are the building blocks of skew lattices. 
Every skew lattice embeds into a direct product of a left-handed skew lattice and a right-handed skew lattice. It is therefore natural to consider the flat coset decompositions, i.e. decompositions of a skew lattice into right and left cosets. In the present paper we discuss such decompositions, their properties and the relation to the coset laws for cancellative and symmetric skew lattices.
\end{abstract}


%
%
%
%
%
%


\section*{Introduction}

Skew lattices can be understood either as non-commutative generalizations of lattices or as double bands,  where by a \emph{band} we mean a semigroup of idempotents. Although noncommutative lattices were introduced by Jordan in \cite{jordan}, and studied later by Cornish \cite{cornish}, the systematical study of modern skew lattice theory began in 1989 by Leech's paper \cite{Le89}, where  {Leech's First} and  Second Decomposition Theorems were proven, revealing the coset structure of a skew lattice. (See Section 1 for exact statements of the theorems. Both decomposition theorems are analogues of basic results in the theory of bands).

In addition to the outer structure revealed by the two decomposition theorems, skew lattices also possess an interesting inner structure, its so called coset structure. Already in the 1989  foundational paper \cite{Le89} certain aspects of the coset structure of a skew lattice were introduced; however, it was fully explored  in  \cite{Le93} where Leech studied what he referred to as \textit{the global geometry of skew lattices}.
The coset structure was later used in \cite{Co09a} to characterize certain sub-varieties of skew lattices, and in \cite{Le11b} for the purpose of studying distributivity in skew lattices, an approach  proposed in \cite{JPC12}. 

By Leech's First Decomposition Theorem a skew lattice is a lattice of its Green's $\DD$-classes which form its maximal rectangular subalgebras. The coset structure provides a picture of \emph{how} different  $\DD$-classes are ``glued'' together, thus providing important additional information.  Given a pair of comparable $\DD$-classes $A>B$, each of the two classes induces a partition of the other class, and the blocks of these partitions are called \textit{cosets}. Much of the internal structure of a skew lattice is determined by these coset decompositions introduced in \cite{Le93} that reveal the interplay of pairs of comparable $\DD$-classes.
All of this background is reviewed in Section \ref{Preliminaries}.

A skew lattice is called \emph{right-handed} if Green's relation $\LL$ is trivial (making $\DD=\RR$), and it is called \emph{left-handed} if Green's relation $\RR$ is trivial (making $\DD=\LL$). (See Section \ref{Preliminaries} for precise definitions). 
By Leech's Second Decomposition Theorem any skew lattice $S$ embeds into the direct product of a right-handed skew lattice $S/\LL$ (called the \emph{right factor} of $S$) and a left-handed skew lattice $S/\RR$ (called the \emph{left factor} of $S$). 
A skew lattice is \emph{flat} if it is either right- or left-handed. 
Flat skew lattices thus form examples of skew lattices that are general enough to reveal structural properties of skew lattices. 
Indeed, it was proven in \cite{Ka05b} that a skew lattice satisfies any identity or equational implication satisfied by both its left factor and its right factor.
But what if just one of the factors has such a property? 
This is addressed for certain cases in Section \ref{Flat coset structure}. 

In Section \ref{Flat coset laws} we explore connections of the flat coset structure of  skew lattices to the study of various important properties, such as cancellation and normality. 
This approach enables us to classify certain varieties of skew lattices. 
The results of Section 3 were motivated by the earlier studies of \cite{Ka08},  \cite{Co09a}, \cite{Le93},  \cite{AAA80}, \cite{Co11} and \cite{JPC12}. 

Basic knowledge on the lattice theory and the semigroup theory can be found in  \cite{Gr71} and \cite{Ho76}, respectively.

\section{Preliminaires}\label{Preliminaries}

A \emph{skew lattice} is an algebra $(S;\lor, \land)$, where $\land $ and $\lor $ are idempotent and associative binary operations that satisfy the absorption laws $x\land (x\lor y)=x=(y\lor x)\land x$ and $x\lor (x\land y)=x=(y\land x)\lor x$.
These identities are equivalent to the \emph{absorption dualities}: $x\wedge y= x$ iff $x\vee y=y$; and $x\wedge y=y$ iff $x\vee y=x$.
The \emph{Green's preorders}  $\preceq_R$ and $\preceq_L$ are defined on $S$ by
\begin{eqnarray*}
x\preceq_R y \text{ iff } x=y\land x\text{ iff } y=y\lor x, \\
x\preceq_L y \text{ iff } x=x\land y  \text{ iff } y= x\lor y.
\end{eqnarray*}
The intersection of $\preceq_R$ and $\preceq L$ is the \emph{natural partial order} given by $x \leq y$ if $x\wedge y = x = y\wedge x$ or equivalently $x\vee y = y = y\vee x$. 
Since $\preceq_L \circ \preceq_R = \preceq_R \circ \preceq_L$, the resulting common outcome $\preceq$ is a preorder, called the \emph{natural preorder} given by: 
$x \preceq y$ iff $x\wedge y\wedge x = x$ or equivalently $y\vee x\vee y = y$.
The induced equivalence relations of $\leq$, $\preceq_L$, $\preceq_R$ and $\preceq$ are denoted respectively by $=$ (of course), $\LL$, $\RR$ and $\DD$. 
Thus $x\Lrel y$ iff $x\wedge y=x$ and $y\wedge x=y$, or dually $x\vee y=y$ and $y\vee x=x$; likewise, $x\Rrel y$ iff $x\wedge y=y$ and $y\wedge x=x$,  or dually $x\vee y=x$ and $y\vee x=y$. 
Finally, $x\Drel y$ iff $x\wedge y\wedge x=x$ and $y\wedge x\wedge y=y$, or dually $x\vee y\vee x=x$ and $y\vee x\vee y=y$. 
The relations $\mathcal L$, $\mathcal R$ and  $\mathcal D$ are known as  the \emph{Green's relations}.

Recall that a band is \emph{right [left] regular} if it satisfies the identity $xyx=yx$ [$xyx=xy$]; is \emph{regular} if it satisfies $xyxzx=xyzx$; and is \emph{rectangular} if it satisfies $xyx=x$.
Skew lattices can be seen as double regular bands as the band reducts $(S,\land )$ and $(S,\lor )$ are regular. 
A skew lattice $ S$ is \emph{rectangular} if and only if  its band reducts $(S,\land )$ and $(S,\lor )$ are rectangular or equivalently if
 $x\land y= y\lor x$ holds.

The following two fundamental theorems hold:

\begin{theorem}[\cite{Le89} Leech's First Decomposition Theorem]\label{1decomp}
Let $S$ be a skew lattice. Then $\DD $ is a congruence, $S/\DD $ is the maximal lattice image of $S$ and the $\DD $-classes of $S$ are its maximal rectangular subalgebras. 
\end{theorem}

\begin{theorem}[\cite{Le89} Leech's Second Decomposition Theorem] \label{2decomp} 
Given a skew lattice $S$, relations $\LL $ and $\RR $ are always congruences.
Moreover, $S/\LL $ is the maximal right-handed image of $S$, $S/\RR $ is the maximal left-handed image of $S$, and the natural projections $S \rightarrow S/\LL$ and $S \rightarrow S/\RR$ jointly yield an isomorphism of $S$ with the fibered product 
$S/\RR \times_{S/\DD } S/\LL = \set{(x,y): x\in S/\RR, y\in S/\LL, p(x)=q(y)}$
where $p:S/\LL \rightarrow S/\DD$ and $q:S/\RR\rightarrow S/\DD$ are the naturally induced homomorphisms.
\end{theorem} 

These are clearly respective skew lattice versions of the Clifford-McLean Theorem and Kimura's fundamental theorem on regular bands. 
That $\LL$ and $\RR$ are full congruences is why skew lattices are called regular (as opposed to just the $\vee$-band and $\wedge$-band reducts being regular). 
An important aspect of regularity is expressed in the following lemma.

\begin{lemma} \label{lem_reg} \cite{Co09a}
Let $S$ be a skew lattice and let $x_1,x_2,u,v$ in $S$ be such
that $u\preceq x_i\preceq v$ for $i\in \{1,2\}$. Then
\[x_1\land v\land x_2=x_1\land x_2 \text{ and } x_1\lor u\lor x_2=
x_1\lor x_2.
\]
\end{lemma}

Consider a skew lattice $S$ with $\DD$-classes $A>B$. Given $b\in B$, the subset $A\land b\land A=\{a\land b\land a \,|\, a\in A\}$ of $B$ is said to be a \emph{coset} of $A$ in $B$ (or an \emph{$A$-coset in $B$}). Similarly, a coset of $B$ in $A$ (or a $B$-coset in $A$) is any subset $B\lor a\lor B =\{b\lor a\lor b \,|\, b\in B \}$ of $A$, for a fixed $a\in A$. On the other hand, given $a\in A$, the \emph{image set} of $a$ in $B$ is the set $a\land B\land a = \set{a \land b\land a\,|\,b\in B}=\set{b\in B\,|\,b< a}.$ Dually, given $b\in B$ the set $b\lor A\lor b = \set{a\in A:b<a}$ is the image set of $b$ in $A$. 

\begin{theorem}\cite{Le93} \label{coset_part}
Let $S$ be a skew lattice with comparable $\DD$-classes $A>B$. Then, $B$ is partitioned by the cosets of $A$ in $B$ and the image set of any element $a\in A$ is a transversal of the cosets of $A$ in $B$; dual remarks hold for any $b\in B$ and the cosets of $B$ in $A$ that determine a partition of $A$. 
Moreover, any coset $B\lor a\lor B$ of $B$ in $A$ is isomorphic to any coset $A\land b\land A$ of $A$ in $B$ under a natural bijection $\varphi $ defined implicitly for any $a\in A$ and $b\in B$ by: $x\in B\lor a\lor B$ corresponds to $y\in A\land b\land A$ if and only if $x\geq y$. 
Furthermore, the operations $\land$ and $\lor$ on $A\cup B$ are determined jointly by the above bijections $\varphi$ and the rectangular structure of each $\mathcal D$-class.
\end{theorem}

Let $S$ be a skew lattice with comparable $\DD$-classes $A>B$ and let $y,y'\in B$. Then $A\land y \land A = A\land y' \land A$ iff for all $x\in A$ the equality $x\land y\land x = x\land y'\land x$ holds.
Dual results hold, having a similar statement (cf. \cite{Co09a}).
All cosets and image sets are rectangular sub skew lattices.
The natural isomorphism between cosets $\varphi:B\lor a\lor B \to A\land b\land A$ mentioned in the theorem, and its inverse $\varphi ^{-1}$, are called \emph{coset bijections}.

A skew lattice is said to be \emph{symmetric} if for all $x,y\in S$, $x\land y= y\land x$ holds if and only if $ x\lor y= y\lor x$ holds. 
$S$ is called \emph{upper symmetric} if $x \land y = y \land x$ implies $x \lor y = y \lor x$; and $S$ is called \emph{lower symmetric} if $x \lor y = y \lor x$ implies $x \land y = y \land x$.

Finally, a skew lattice $S$ is called \emph{cancellative} if for all $x,y,z\in S$, 
$z\lor x=z\lor y \text{ and } z\land x=z\land y \text{ imply } x=y,$ and
$x\lor z=y\lor z \text{ and } x\land z=y\land z \text{ imply } x=y.$
Cancellative skew lattices are always symmetric, see \cite{Ka08}. They are also \emph{quasi-distributive} in that their maximal lattice images are distributive. A skew lattice is \emph{upper cancellative} if it is upper symmetric and simply cancellative.
Dually, a skew lattice is \emph{lower cancellative} if it is lower symmetric and simply cancellative.

\begin{theorem}\label{pl_canc}\cite{Co09a}
Let $S$ be a quasi-distributive, symmetric skew lattice. Then $S$ is cancellative iff one (and hence both) of the following equivalent statements hold:
\begin{itemize}
\item[(i)] given any skew diamond $\set{J>A,B>M}$ in $S$ and any $x,x'\in A$, $M\lor x\lor M = M\lor x'\lor M$ holds if and only if $B \lor x\lor B = B \lor x'\lor B$ holds;
\item[(ii)] given any skew diamond $\set{J>A,B>M}$ in $S$ and any $x,x'\in A$, $B\land x \land B=B\land x' \land B$ holds if and only if $J\land x \land J=J\land x' \land J$ holds.
\end{itemize}
\end{theorem}

Recall from \cite{Ka08} that a skew lattice $S$ is \emph{right cancellative} if for all $x,y,z\in S$ the pair of equalities $x\lor z=y\lor z$ and $x\land z=y\land z$ implies $x=y$. \emph{Left cancellative} skew lattices are defined dually. 
A skew lattice is \emph{simply cancellative} if for all $x,y,z\in S$ the pair of equalities $x\lor z\lor x=y\lor z\lor y$ and $x\land z\land x=y\land z\land y$ implies $x=y$.
Clearly, cancellative skew lattices are the ones that are simultaneously right cancellative and left cancellative. If $S$ is symmetric then right cancellation is equivalent to left cancellation and thus coincides with (full) cancellation. (See \cite{Ka08}.)

Recall that a skew lattice is said to be \emph{normal} if it satisfies the identity $x\land y\land z\land w= x\land z\land y\land w$ and, dually, it is named \emph{conormal} if it satisfies $x\lor y\lor z\lor w= x\lor z\lor y\lor w$, cf. \cite{Le92}.

\begin{lemma}\label{pl_normal}\cite{Co11} 
Let $S$ be a skew lattice. Then $S$ is normal iff for each pair of comparable $\DD $-classes $A>B$ in $S$, the class $B$ is an entire coset of $A$ in $B$. That is, 
\begin{center}
$A\land x\land A=A\land x'\land A$
\end{center}
holds for all $x,x'\in B$.
Dually, $S$ is conormal iff $B\lor x\lor B=B\lor x'\lor B$ holds for all pairs of comparable $\DD $-classes $A>B$ in $S$ and all $x,x'\in A$.
\end{lemma}

\section{Flat coset structure}
\label{Flat coset structure}

A \emph{right coset} of $A$ in $B$ is any set of the form $b\land A=\{b\land a\,|\, a\in A\}$, where $b\in B$. 
Similarly, a \emph{right coset} of $B$ in $A$ is any set of the form $B\lor a$ for $a\in A$.
Left cosets are defined analogously. We say that a coset is \emph{flat} whenever it is a right coset or a left coset. 

\begin{lemma}\label{strg_prop_right}
Given comparable $\DD$-classes $A>B$ in a skew lattice $S$ and $y,y'\in B$, the following are equivalent:
\begin{itemize}
\item[(i)] $y \land A = y' \land A$;
\item[(ii)] $y\land x =  y'\land x$, for all $x\in A$;
\item[(iii)]  $y\land x = y'\land x$, for some $x\in A$;
\item[(iv)]  $y\land x = y'\land x'$, for some $x,x'\in A$.
\end{itemize}
Similar equivalences hold for left cosets of $A$ in $B$, and for left or right cosets of $B$ in $A$.
\end{lemma}

\begin{proof}
Clearly (ii) implies (i) and (iii), each of which implies (iv). Conversely, let (iv) hold and take any $a\in A$. By (iv) there exist $x,x'\in A$ s.t. $y\land x=y'\land x'$. Using Lemma \ref{lem_reg} we get 
\[y\land a=y\land x\land a=y'\land x'\land a=y'\land a
\]
which implies (ii). 
\end{proof}

\begin{lemma}\label{lema2}
The right cosets of $A$ in $B$ partition $B$. In detail, $b\in b\land A$ for all $b\in B$, and for all $x$ in $B$, $x\in b\land A$ is equivalent to $x\land A=b\land A$. Moreover, the partition of $B$ by right cosets of $A$ in $B$ refines its partition by cosets of $A$ in $B$. Similar assertions hold for left cosets of $A$ in $B$ and also for left or right cosets of $B$ in $A$.  
\end{lemma}

\begin{proof}
Given $b\in B$, for all $a\in A$, $b=b\land (b\lor a\lor b)\in b\land A$. Also, if $x=b\land a$ for some $a\in A$, then for all $y\in A$, regularity gives $x\land y=b\land a\land y=b\land y$ so that $x\land A=b\land A$. If the latter holds, $x\in b\land A$ follows from Lemma \ref{strg_prop_right}. Clearly $b\land A=(b\lor a\lor b)\land b\land A\subseteq A\land b\land A$.
\end{proof}

\begin{lemma}\label{fullright}
Let $S$ be a skew lattice with comparable $\DD$-classes $A>B$ and $x,y\in B$, $u,v\in A$.
Then:
\begin{itemize}
\item[(i)] $x\land A=y\land A$ if and only if $A\land x\land A=A\land y\land A$ and $x\Rrel y$;
\item[(ii)] $A\land x=A\land y$ if and only if $A\land x\land A=A\land y\land A$ and $x\Lrel y$;
\item[(iii)] $B\lor u=B\lor v$ if and only if $B\lor u\lor B=B\lor v\lor B$ and $u\Rrel v$; 
\item[(iv)] $u\lor B=v\lor B$ if and only if $B\lor u \lor B=B\lor v\lor B$ and $u\Lrel v$.
\end{itemize}
\end{lemma}

\begin{proof}
Again, we need only prove (i) since (ii)--(iv) then follow from standard dualities. Given $x\land A=y\land A$, clearly $A\land x\land A=A\land y\land A$. Lemmas \ref{strg_prop_right} and \ref{lema2}, also give $x=y\land a$ and $y=x\land a'$ for some $a,a'\in A$, from which $x\Rrel y$ follows. Conversely, given $A\land x\land A=A\land y\land A$ and $x\Rrel y$, regularity implies $x\land A=x\land A\land x\land A=x\land A\land y\land A=x\land y\land A=y\land A$.
\end{proof}

It follows that right and left cosets are $\mathcal R$-classes and $\mathcal L$-classes in full cosets.
To fully grasp the import of this we recall the following result of Kinyon, Leech and Pita Costa  \cite{Le11b} Theorem 4.2.)

\begin{proposition} \label{prop4}
Given comparable $\DD$-classes $A>B$, the partition of $B$ by $A$-cosets is a congruence partition of $B$ and the partition of $A$ by $B$-cosets is a congruence partition of $A$. Thus all $A$-cosets in $B$ are rectangular subalgebras of $B$ and all $B$-cosets in $A$ are rectangular subalgebras of $A$. Finally, all coset bijections between them are isomporphisms of the respective rectangular subalgebras.
\end{proposition}

A primitive skew lattice $A>B$ is thus a union of \emph{binormal} primitive skew lattices $A'>B'$ arising as pairs of cosets $A'$ in $A$ and $B'$ in $B$. The subalgebras $A'\cup B'$ 
of all such pairs are isomorphic; and given any two such pairs $A'>B'$ and $A''>B''$, either $A'=A''$ and $B'\cap B''=\emptyset$, or $B'=B''$ but $A'\cap A''=\emptyset$, or else $A'\cap A''=\emptyset=B'\cap B''$. In detail:

\begin{theorem}\label{th5}
Given $\DD$-classes $A>B$ with $a\in A$ and $b\in B$, the cosets $A'=B\lor a\lor B$ and $B'=A\land b\land A$ form a binormal primitive subalgebra of $A\cup b$ that is isomorphic to $\mathbf 2\times A'$, or equivalently, to $\mathbf 2\times B'$. For this subalgebra $A'\cup B'$ the following hold:
\begin{itemize}
  \item[(i)] Both $\DD$-classes $A'>B'$ are full cosets of the other and the coset bijection from $A'$ to $B'$ is an isomorphism.
  \item[(ii)] The $\RR$-classes of $B'$ are the right cosets $b'\land A$ of $A$ lying in $B'$ and the $\mathcal L$-classes of $B'$ are the left cosets $A\land b'$ of $A$ lying in $B'$. Dually, the $\RR$-classes of $A'$ are the right cosets $B\lor a'$ of $B$ lying in $A'$ while the $\LL$-classes are the left cosets $a'\lor B$ of $B$ lying in $A'$.
  \item[(iii)] All right cosets in $A'\cup B'$ are isomorphic. Given right cosets $B\lor a'$ in $A'$ and $b'\land A$ in $B'$, a natural isomorphism from $B\lor a'$ to $b'\land A$ is given by $x\mapsto b'\land x$. The inverse isomorphism is given by $y\mapsto y\lor a'$. Under this correspondence, $b'\land x$ is the unique element $y$ in $b'\land A$ such that $y\preceq_L x$; inversely $y\lor a'$ is the unique $x$ in $B\lor a'$ such that $y\preceq_L x$. 
  \item[(iv)] Dual remarks hold for left cosets in $(B\lor a\lor B) \cup (A\land b\land A)$.
  \item[(v)] Finally, all left [right, $2$-sided] cosets of $A$ in $B'$ are also $A'$-cosets in $B'$ of the same type and all left [right, $2$-sided] cosets of $B$ in $A'$ are also $B'$-cosets in $A'$ of the same type.
\end{itemize}
\end{theorem}

\begin{proof}
This follows from the previous results, the utter triviality of a binormal primitive skew lattice, the fact that left translations preserve the $\mathcal R$-relation (bijectively in the rectangular and also in the binormal primitive case) and the fact that bijections between $\mathcal R$-classes are isomorphisms.
\end{proof}
\begin{corollary}\label{cor5}
All right cosets in $A\cup B$ are isomorphic, as are all left cosets. Given right cosets $B\lor a$ in $A$ and $b\land A$ in $B$, a natural isomorphism of $B\lor a$ with $b\land A$ is given by $x\mapsto b\land x$; the inverse isomorphism is given by $y\mapsto y\lor a$. In this context, $b\land x$ is the unique element $y$ in $b\land A$ such that $y\preceq_L x$; inversely $y\lor a$ is the unique $x$ in $B\lor a$ such that $y\preceq_L x$.
\end{corollary}

\begin{corollary}\label{cor7}
Let $A>B$ be comparable $\DD$-classes in a skew lattice $S$ and let $y,y'\in A$ and $x,x'\in B$. Then the intersection $(x\land A)\cap (A\land x')$ is nonempty if and only if $A\land x\land A=A\land x'\land A$, in which case $(x\land A)\cap (A\land x')=\{x\land x'\}$.  Dually, $(y\lor B)\cap (B\lor y')$ is nonempty if and only if $B\lor y\lor B=B\lor y'\lor B$, in which case $(y\lor B)\cap (B\lor y')=\{y\lor y'\}$.
\end{corollary}

\begin{proof}
Clearly $A\land x\land A=A\land x'\land A$ is necessary for a pair of their subsets to have nonempty intersection. But when this occurs, $(x\land A)\cap (A\land x')=\RR_x\cap \LL_{x'}=\{x\land x'\}$ in the $\DD$-class $A\land x\land A$ of the primitive subalgebra.
\end{proof}

\begin{corollary}\label{fullright2}
Let $S$ be a skew lattice with comparable $\DD$-classes $A>B$ and $x,y\in B$.
The following statements are equivalent:
\begin{itemize}
\item[(i)] $A\land x\land A=A\land y\land A$; 
\item[(ii)] $A\land x\land y=A\land y$ and $x\land A=x\land y\land A$.
\end{itemize}
The dual result holds for $B$-cosets in $A$.
\end{corollary}

\begin{proof}
The direct implication  is an easy consequence of regularity. E.g., given $(i)$, one has
\[
A\land y=A\land y\land A\land y=A\land x\land A\land y=A\land x\land y
\] 
with regularity being used in the final equality. The converse is trivial.
\end{proof}

\begin{corollary}\label{cosetdirectprod}
Let $S$ be a skew lattice with two comparable $\DD $-classes $A>B$. 
Then, for all $x\in B$, there exists an isomorphism 
\[\delta_{A\land x\land A}:A\land x\land A\rightarrow (A\land x)\times (x\land A).
\]
Dually, for all $y\in A$, there exists an isomorphism 
\[
\delta_{B\lor y\lor B}:B\lor y\lor B\rightarrow (y\lor B)\times (B\lor y).
\]
\end{corollary}

\begin{proof}
We view $A\land x\land A$ as the bottom $\DD$-class of a primitive subalgebra. The corollary asserts that an isomorphism $\delta$ of $A\land x\land A$ with $\LL_x\times \RR_x$ exists. But this is true for any rectangular algebra $(D,\lor, \land)$ and any fixed element $x$ in $D$. Defining $\delta$ by $\delta(x)=(y\land x,x\land y)$ gives a $\land$-isomorphism that is necessarily a $\lor$-isomorphism also. The dual case is similar.
\end{proof}

Given sets $X$, $Y$, $Z$ and $W$  and maps $f:X\rightarrow Z$, $g:Y\rightarrow W$ we denote by $f\times g$ the map from $X\times Y$ to $Z\times W$ that assigns to each pair $(x,y)\in X\times Z$ the pair $(f(x),g(y))$. The following result is an immediate corollary to the Corollary \ref{cosetdirectprod}.

\begin{corollary}\label{inner_kimura} 
Let $S$ be a primitive skew lattice with two comparable $\DD $-classes $A>B$. Given $a\in A$ and $b\in B$, $a>b$, consider the maps:
\[
\begin{array}{rcl}
\varphi_{a,b}: B\lor a\lor B & \to & A\land b\land A \\
x & \mapsto & x\land b\land x
\end{array},
\,
\begin{array}{rcl}
\varphi_{a,b}^L: a\lor B & \to & A\land b\\
x & \mapsto & x\land b
\end{array},
\,
\begin{array}{rcl}
\varphi_{a,b}^R: B\lor a & \to &  b\land A \\
x & \mapsto &  b\land x.
\end{array}
\]
Then the following is a commutative diagram of skew lattice isomorphisms:
\begin{center}
\begin{tikzpicture}
\path (-1.5,1.5) node[] (S) {$B\lor a\lor B$};

\path (-1.5,-1.5) node[] (L) {$A\land b\land A$};

\path (3.5,1.5) node[] (R) {$(a\lor B)\times (B\lor a)$};

\path (3.5,-1.5) node[] (D) {$(A\land b)\times (b\land A)$};

\draw[arrows=-latex'] (S) -- (R) node[pos=.5,above] {$\delta_{B\lor a\lor B}$};
\draw[arrows=-latex'](S) -- (L) node[pos=.5,left] {$\varphi_{a,b}$};
\draw[arrows=-latex'] (L) -- (D) node[pos=.5,above] {$\delta_{A\land b\land A}$};
\draw[arrows=-latex'](R) -- (D) node[pos=.5,right] {$\varphi_{a,b}^L\times \varphi_{a,b}^R$}; 	 
\end{tikzpicture}
\end{center}
\end{corollary}

The assertion of Corollary \ref{inner_kimura} is even true when $a>b$ does not hold.

%
Let $S$ be a skew lattice with two comparable $\DD $-classes $A>B$. Given $b\in B$ the \emph{left image set} of $b$ in $A$ is the set $b\lor A=\{b\lor a\,|\, a\in A\}$; given $a\in A$ the \emph{left image set} of $a$ in $B$ is the set $B\land a=\{b\land a\,|\, b\in B\}$. Right image sets are defined dually. The left image sets $b\lor A$  and $B\land a$ are both contained in $\LL$-classes in their respective $\DD$-classes and $b\preceq_L b\lor a$, $b\land a\preceq_L a$ hold.  The left image set of any element $a\in A$ in $B$ forms a transversal of the family of all right cosets of $A$ in $B$ as $(B\land a)\cap (b\land A)=\{b\land a \}$. Hence all left image sets are equipotent. Moreover, any particular $\LL$-class naturally parametrizes the $\RR$-classes of that $\DD$-class: $\{\RR_x \,|\, x \in \LL_u \}$ for any $u$ in that $\DD$-class. If $S$ is a binormal primitive skew lattice with $\DD$-classes $A>B$ then given $x\in A$ and $y\in B$, $y\preceq_L x$ holds if and only if $y$ lies in the $\LL$-class $\LL_{x\land b\land x}=B\land x$. In this case $x\land b\land x$ is the unique image of $x$ in $B$.


\section{Flat Coset Laws}
\label{Flat coset laws}

The following results show the impact of the flat coset decomposition on the coset laws for cancellative skew lattices and for normal skew lattices.

We shall now turn our attention to normal skew lattices and corresponding coset laws.
The relation with quasi-normality shall also be discussed.

\begin{proposition}\label{cs_normal} 
Let $S$ be a skew lattice. Then $S$ is normal iff for each comparable pair of $\DD $-classes $A>B$ in $S$ and all $x,x'\in B$ the following pair of implications hold:
\begin{itemize}
\item[(i)] if $x\LL x'$ then $A\land x =A\land x'$;
\item[(ii)] if $x\RR x'$ then $x\land A=x'\land A$.
\end{itemize}
Dually, $S$ is conormal iff for all comparable pairs of $\DD $-classes $A>B$ in $S$ and for all $x,x'\in A$ the following pair of implications hold:
\begin{itemize}
\item[(iii)] if $x\RR x'$ then $B\lor x=B\lor x'$;
\item[(iv)] if $x\LL x'$ then $x\lor B =x'\lor B$.
\end{itemize}

\end{proposition}

\begin{proof}
First assume that $S$ is normal. If $x\Lrel x'$ then by Lemma \ref{fullright} and Lemma \ref{pl_normal} we obtain:
\[
A\land x=(A\land x\land A)\cap \LL_{x}=(A\land x'\land A)\cap \LL_{x'}=A\land x'
\]
which proves (i).
The proof for (ii) is similar.

Conversely, assume that (i) and (ii) hold, and let $A>B$ be comparable $\DD $-classes in $S$. 
Take $x,x'\in B$.
As $B$ is rectangular there exists $z\in B$ such that $x\Rrel z$ and $z\Lrel x'$.
By the assumption we have $x\land A=z\land A$ and $A\land z=A\land x'$, and
thus $A\land x\land A=A\land z\land A=A\land x'\land A$.
Hence $S$ is a normal skew lattice by Lemma \ref{pl_normal}.
The statement regarding conormal skew lattices has a similar proof.
\end{proof}

\begin{definition}
A skew lattice is \emph{right quasi-normal} (RQN) if it satisfies the identity $y\land x\land a=y\land a\land x\land a$, and it is \emph{left quasi-normal} (LQN) if it satisfies the identity $a\land x\land y=a\land x\land a\land y$. Equivalently, right [left] quasi-normal skew lattices are the ones for which $(S;\land)$ is a right [left] quasi-normal band. 
These bands are defined in \cite{Pe71}. Dual definitions determine \emph{[left] right quasi conormal} skew lattices.
The following results provide us with useful characterizations of such algebras.
\end{definition}

\begin{remark}
Recall that for regular bands, $S/\RR$ is left normal if and only if $S$ is right quasi-normal (cf. \cite{Pe71} Th. 6).
\end{remark}

\begin{lemma}\label{lemma-leq-ideal}
Let $S$ be a skew lattice and $y,x\in S$. Then  the sets $x\land S$ and $S\land x$ are subalgebras of $S$, and the following hold:
 \begin{enumerate}
   \item $y\preceq_{R} x$ if and only if $y\in x\land S$;
   \item $y\preceq_{L} x$ if and only if $y\in S\land x$.
  \end{enumerate}
\end{lemma}

\begin{proof}
In order to prove (1) let $y\preceq_{R} x$. Then $y=x\land y\in x\land S$. Conversely, assume $y\in x\land S$. Then $y=x\land u$ for some $u\in S$. Thus $x\land y=x\land x\land u=y$ which implies $y\preceq_{R} x$. This proves (1) and (2) follows by a dual argument. It remains to prove that $x\land S$ and $S\land x$ are subalgebras of $S$.
Take $u, v\in x\land S$. By (1) the elements $u$ and $v$ are of the form $u=x\land u$ and $v=x\land v$. Hence 
\[x\land u\land v=u\land v
\]
and thus $u\land v\in x\land S$. By Lemma \ref{lemma-leq-ideal}  in order to prove that also $u\lor v\in x\land S$ it suffices to show that $u\lor v\preceq_{R} x$ which is equivalent to $x=x\lor u\lor v$. This is indeed the case as:
\[x\lor u\lor v=x\lor (x\land u)\lor (x\land v) 
\]
which equals $x$ by absorption.
\end{proof}

\begin{proposition}\label{quasinormal}
A skew lattice $S$ is right quasi-normal if and only if for all $x\in S$ the factor algebra $(x\land S)/\RR$ is a lattice,or equivalently, $S$ is right quasi-normal iff for all $ x\in S$, the subalgebra $x\land S$ is right-handed. 
Dually, $S$ is left quasi-normal if and only if for all $x\in S$ the factor algebra $(S\land x)/\LL$ is a lattice, or equivalently for all $ x\in S$, the subalgebra $S\land x$ is left-handed. 
\end{proposition}

\begin{proof}
Assume that $S$ is right quasi-normal and let $x\in S$, $y,y'\in x\land S$ be such that $y\Lrel y'$. Then $y=x\land y$ and $y'=x\land y'$ by Lemma \ref{lemma-leq-ideal}.
Thus: 
\[
y=x\land y=x\land y\land y'=x\land y'\land y\land y'=x\land y'=y',
\]
 where the second and forth equality follow by $y\Lrel y'$, and the third equality follows by right quasi-normality.

Conversely, assume that $(z\land S)/\RR$ is a lattice for all $z\in S$, and take arbitrary $y,x,a\in S$.  
Consider $x\land S$ that is a subalgebra bx Lemma \ref{lemma-leq-ideal}.
By regularity we have:
\[
(x\land y\land a)\land (x\land a\land y\land a)=(x\land y\land a)\land (x\land y\land a)=x\land y\land a
\]
 and 
 \[
 (x\land a\land y\land a)\land (x\land y\land a)=(x\land a\land y\land a)\land (x\land a\land y\land a)=x\land a\land y\land a.
 \]
Thus $(x\land a\land y\land a)\Lrel (x\land y\land a).$
However, as by the assumption all $\LL$-classes of $x\land S$ are trivial, $x\land y\land a=x\land a\land y\land a$ follows.
The proof of the dual statement is similar. 
\end{proof}

The next result relates Propositions \ref{cs_normal} and \ref{quasinormal}, giving us a characterization for left [right] quasi-normal skew lattices of coset nature.

\begin{proposition}\label{cosetquasinormal}
Let $S$ be a skew lattice. Then,  

\begin{itemize}
\item[(i)] $S$ is left quasi-normal if and only if for all comparable $\DD$-classes $A>B$ in $S$ and $x,x'\in B$ such that $x\Rrel x'$, then $x\land A=x'\land A$. 
\item[(ii)] $S$ is right quasi-normal if and only if for all comparable $\DD$-classes $A>B$ in $S$ and $x,x'\in B$ such that $x\Lrel x'$, then $A\land x=A\land x'$; 
\end{itemize}

Dual results hold for conormality.
\end{proposition}

\begin{proof}
Let $a\in A$ and $x,x'\in B$ such that $x\Rrel x'$.
Due to the hypothesis and the fact that $x\Drel x'$, $x\land a=x'\land x\land a=x'\land x\land x'\land a=x'\land a$ so that $x\land A=x'\land A$ as required by (i). 
Conversely, let $x\in y\land S$ and consider $A=\DD_{y}$.
Let $x'\in S$ such that $x\Rrel x'$ 
Then $x'=x\land x'=y\land x\land x'=y\land x'\in y\land S$.
Then, the hypothesis implies that $A\land y\land x=A\land y\land x'$. 
As $y\in A$ then Lemma \ref{strg_prop_right} implies that 
\[
x=y\land x=y\land y\land x=y\land y\land x'=y\land x'=x'.
\]
Hence, Proposition \ref{quasinormal} implies that $S$ is right quasi-normal.
The proof of (ii) is similar.
\end{proof}

The following result is a consequence of Propositions \ref{cs_normal} and \ref{cosetquasinormal}.
In fact, it also follows the research made for bands of semigroups in \cite{Pe71} when considering the reducts $(S;\wedge)$ and $(S;\vee)$ of a skew lattice $S$. 

\begin{corollary}
Let $S$ be a skew lattice. Then, $S$ is normal if and only if $S$ is simultaneously right quasi-normal and left quasi-normal. 
Dually, $S$ is conormal if and only if $S$ is simultaneously right quasi-conormal and left quasi-conormal. 
\end{corollary}

In the remainder of the paper we will give a further insight to the flat coset decomposition of cancellative skew lattices for which the lattice image is distributive, and therefore the ones permitting the coset laws established in \cite{Co09a}.

\begin{remark}\label{imp}
Recall that given a skew diamond $\set{J>A,B>M}$ and elements $x,x'\in A$, the equality $M\lor x\lor M=M\lor x'\lor M$ always implies $B\lor x\lor B=B\lor x'\lor B$. Likewise, the equality $J\land x\land J=J\land x'\land J$ implies $B\land x\land B=B\land x'\land B$. Proposition \ref{strg_lema_right} below is a flat version of this result.
\end{remark}

\begin{proposition}\label{strg_lema_right}
Let $S$ be a skew lattice and $\set{J>A,B>M}$ a skew diamond in $S$. 
Given any $x,x'\in A$ the following hold:

\begin{itemize}
\item[(i)] if $ M \lor x= M\lor x'$ then $ B \lor x= B\lor x'$;
\item[(ii)] if $ x\land J = x'\land J$ then $x\land B= x'\land B$.
\end{itemize}

Similar remarks hold regarding left cosets.
\end{proposition}

\begin{proof}
We will prove (i) having in mind that (ii) follows by a dual argument.
Let $x,x'\in A$ and assume that $M \lor x = M\lor x'$. Lemma \ref{strg_prop_right} implies the existence of $m\in M$ such that $ m \lor x= m\lor x'$. 
Let $b\in B$ be such that $m\preceq_{L} b$. Then, 
\[
\begin{array}{rcl}
b\lor x & = & b\lor m\lor x \\
              & = & b\lor m\lor x'  \\
              & = & b\lor x'.
\end{array}
\] 
Lemma \ref{strg_prop_right} then implies $B\lor x = B\lor x'$.
\end{proof}

\begin{proposition}\label{cosidentities}
Let $\textbf S$ be a skew lattice. 
Then given any skew diamond $\set{J>A,B>M}$ in $\textbf S$ and any $x,x'\in A$ the following equivalences hold:
\begin{itemize}
\item[(i)] ($M\lor x\lor M= M\lor x'\lor M \Leftrightarrow B \lor x \lor B= B \lor x'\lor B$) if and only if ($M\lor x= M\lor x' \Leftrightarrow B \lor x = B \lor x'$ and $x\lor M = x'\lor M\Leftrightarrow x\lor B = x'\lor B$);
\item[(ii)] ($B\land x \land B=B\land x' \land B \Leftrightarrow J\land x \land J=J\land x' \land J$) if and only if ($x \land B=x' \land B \Leftrightarrow x \land J=x' \land J$ and $ B \land x=B \land x' \Leftrightarrow J \land x=J \land x'$).
\end{itemize}
\end{proposition}

\begin{proof}
We will only show (i) as (ii) has an analogous proof. 
By Proposition \ref{strg_lema_right} and the comment above it, all the direct implications of the considered equivalences always hold. So, only the converse implications need to be addressed.
Let $\set{J>A,B>M}$ be a skew diamond in $S$ and $x,x'\in A$. First assume that $M\lor x\lor M= M\lor x'\lor M \Leftrightarrow B \lor x \lor B= B \lor x'\lor B$ holds. If $B \lor x = B \lor x'$ then
 Lemma \ref{fullright} implies $B \lor x \lor B= B \lor x'\lor B$ and $x\Rrel x'$.
Hence $M\lor x\lor M= M\lor x'\lor M$ and $x\Rrel x'$ by the assumption, and thus $M\lor x= M\lor x'$ follows by Lemma \ref{fullright}.

Conversely, assume that both $M\lor x= M\lor x' \Leftrightarrow B \lor x = B \lor x'$ and $x\lor M = x'\lor M\Leftrightarrow x\lor B = x'\lor B$ hold. If $B\lor x\lor B=B\lor x'\lor B$ then by
 Proposition \ref{fullright2} there exists $y\in B\lor x\lor B$ such that $B\lor y=B\lor x$ and $y\lor B=x'\lor B$.
Proposition \ref{strg_lema_right} then implies $M\lor y=M\lor x$ and $y\lor M=x'\lor M$.
Thus $M\lor x\lor M=M\lor y\lor M=M\lor x'\lor M$ follows.
\end{proof}

\begin{proposition}\label{upsymcl}
Let $ S$ be a skew lattice such that $S/\DD$ is a distributive lattice.
\begin{itemize}
\item[(i)] if $S$ is lower symmetric then $ S$ is lower cancellative if and only if $M\lor x\lor M= M\lor x'\lor M \Leftrightarrow B \lor x \lor B= B \lor x'\lor B$ holds for all skew diamonds $\set{J>A,B>M}$ in $ S$ and all $x,x'\in A$.
\item[(ii)] if $S$ is upper symmetric then $ S$ is upper cancellative if and only if given any skew diamond $\set{J>A,B>M}$ in $ S$ and any $x,x'\in A$, $B\land x \land B=B\land x' \land B \Leftrightarrow J\land x \land J=J\land x' \land J$ holds.
\end{itemize}
\end{proposition}

\begin{proof}
We will now prove $(i)$. The proof of $(ii)$ is similar. 

Let $\set{J>A,B>M}$ be a skew diamond in $S$. By Remark \ref{imp} the direct implication always holds.
So, let $x,x'\in A$ be such that $B \lor x \lor B= B \lor x'\lor B$ and
suppose that $M\lor x\lor M\neq M\lor x'\lor M$.
Let $m_0\in M$. Consider $u=m_0 \vee x\vee m_0$ and $v=m_0 \vee x'\vee m_0$.
There exists $b_0\in B$ such that $b_0>m_0$.
Then, $b_0\vee u\vee b_0=b_0\vee x\vee b_0=b_0\vee x'\vee b_0=b_0\vee v\vee b_0$, where the second equality is due to the assumption that $B \lor x \lor B= B \lor x'\lor B$, and
thus $u< b_0\vee u\vee b$ and $v< b_0\vee v\vee b_0$.
Therefore $m_0<u,v,b_0<b_0\vee u\vee b_0$ determine a copy of $NC_5$ and hence contradicts the assumption that $ S$ is simply cancellative. 

Conversely, if $ S$ is not lower cancellative (ie. it is not simply cancellative, since it is lower symmetric by the assumption), 
then by a result of \cite{Ka08} $ S$ contains a subalgebra
${S}'$ isomorphic to $\mathbf{NC}_5$, given by the diagram
below. (In $\mathbf{NC}_5$ operations
on $x_1$ and $x_2$ can be defined in two ways: for $i,j\in \{1,2\}$
either $x_i\wedge x_j=x_j$ and $x_i\vee x_j=x_i$ which yields a 
right-handed
structure, or $x_i\wedge x_j=x_i$ and $x_i\vee x_j=x_j$ yielding a
left-handed
structure.) Let $A$, $B$, $M$ and $J$ denote the $\DD$-classes
of elements $x_1$, $y$, $u$ and $v$ in $ S$, respectively.
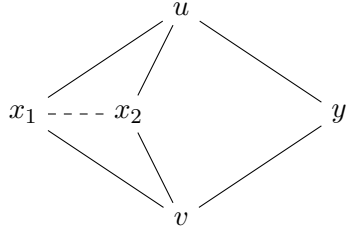
\[
\begin{tikzpicture}[scale=.7]
  \node (1) at (0,2){$u$} ;
  \node (a) at (-3,0){$x_1$} ;
  \node (b) at (-1,0){$x_2$};
  \node (c) at (3,0){$y$}  ;
  \node (0) at (0,-2){$v$} ;
  \draw (1) -- (c) -- (0) -- (a) -- (1) -- (b) -- (0);
 \draw[dashed] (a) -- (b);
\end{tikzpicture}
\]
Since $x_1$ and $x_2$ are both contained in the image of $u$ in $A$,
they 
cannot lie in
the same coset of $M$ in $A$. On the other hand, $B\vee x_1\vee B$ and
$B\vee x_2\vee B$
both contain $v$ and hence coincide by Theorem \ref{coset_part}.
\end{proof}

\begin{proposition}\label{pl_canc_left}
Let $S$ be a symmetric skew lattice such that $S/\DD$ is a distributive lattice. 
Then, the following statements are equivalent:
\begin{itemize}
\item[(i)] $ S/\RR$ is cancellative;
\item[(ii)] given any skew diamond $\set{J>A,B>M}$ in $S$ and any $x,x'\in A$, $M\lor x = M\lor x'$ holds if and only if $B\lor x = B\lor x'$ holds;
\item[(iii)] given any skew diamond $\set{J>A,B>M}$ in $S$ and any $x,x'\in A$, $x \land B=x' \land B$ holds if and only if $x \land J=x' \land J$ holds.
\end{itemize}
A dual result holds regarding right cosets in the skew lattice $S$.
\end{proposition}

\begin{proof}
We will show that $(i) \Leftrightarrow (ii)$ using the characterization of Theorem \ref{pl_canc}.
The equivalence $(i) \Leftrightarrow (iii)$ is proved similarly.

Let $\set{J>A,B>M}$ be a skew diamond in $ S$.
Assume that $S/\RR$ is cancellative. 
Due to Lemma \ref{strg_lema_right} we need only to show that $B \lor x = B \lor x'$ implies $M \lor x = M \lor x'$, for all $x,x'\in A$.
As $S/\RR$ is cancellative and left-handed, all the cosets in $S/\RR$ are left cosets and thus
\[
M_L\lor x_L = M_L\lor x'_L \Leftrightarrow B_L\lor x_L = B_L\lor x'_L
\]
Let $x,x'\in A$ such that $B \lor x = B \lor x'$. Then, Proposition \ref{coset-full-right} implies that 
\begin{align}
B\vee x=B\vee x' & \Rightarrow x_R=y_R \text{  and  } B_L\vee x_L=B_L \vee x'_L \\
			   & \Rightarrow x_R=y_R \text{  and  } M_L\vee x_L=M_L \vee x'_L \\
			   & \Rightarrow M\vee x=M\vee x'. 
\end{align}

Conversely, assume that $M\lor x = M\lor x'$ if and only if $B\lor x = B\lor x'$, for all skew diamonds $\set{J>A,B>M}$ in $S$ and all $x,x'\in A$.
Then, $M_L\lor x_L = M_L\lor x_L'$ if and only if $B_L\lor x_L = B_L\lor x'_L$, for all skew diamonds $\set{J>A,B>M}$ in $S/\RR$ and all $x_L,x'_L\in A_L$.
As $S/\RR$ is a left-handed skew lattice, all its cosets are left cosets and, therefore, $S/\RR$ is cancellative due to Theorem \ref{pl_canc}. 
\end{proof}

Proposition \ref{pl_canc_left} above leads us to define the following notions which are not to be confused with left and right cancellation as defined in the preliminary section.

\begin{definition}
A \emph{left-coset cancellative} skew lattice is a skew lattice $ S$ such that $ S/\RR$ is cancellative.
Dually, a \emph{right-coset cancellative} skew lattice is a skew lattice $ S$ such that $ S/\LL$ is cancellative. 
Due to \cite{Ka05b} both of these classes of algebras constitute varieties. 
\end{definition}

As it was proved in \cite{Ka05b} that a skew lattice $S$ satisfies any identity that is satisfied by both its left factor $S/\RR$ and its right factor $S/\LL$, the following result is a direct consequence of the definitions.

\begin{corollary}
A skew lattice is cancellative if and only if it is both right-coset cancellative and left-coset cancellative.
\end{corollary}

The result of Proposition \ref{pl_canc_left} provides us with a deeper insight on the coset structure of cancellative skew lattices and new subclasses determined by the corresponding laws for flat cosets.
These achievements close the section and the paper.
Several aspects of research on the combinatorial consequences of such coset decomposition can be found in \cite{JPC12}.
Furthermore, the impact of the flat coset structure in other coset laws regarding strictly categorical or distributive skew lattices as in \cite{Co11}, \cite{Le11a} or \cite{Le11b} are a matter of research that we will address to in the future.

\section{Example on matrices in a ring}
We conclude the paper with a demonstration of the coset concepts in the case of skew lattices in rings of matrices. We begin this final section  with a couple of technical results.

Let $S$ be a skew lattice and let
\[\begin{array}{rcl}
\varphi: S & \to &S/\Rrel \times_{S/\Drel} S/\Lrel \\
x & \mapsto & (x_L, x_R)
\end{array}
\]
be the isomorphism from Theorem \ref{2decomp}. 
Given a $\DD$-class $D$ in $S$ denote $D_L=\{ x_L\,|\, (x_L,x_R)=\varphi (x) \text{ for some } x\in D\}$ and $D_R=\{ x_R\,|\, (x_L,x_R)=\varphi (x) \text{ for some } x\in D\}$.
The following lemma is a direct consequence of Theorem \ref{2decomp}.

\begin{lemma}\label{fullcosets-left-right}
Let $x,y\in A$ and $u,v\in B.$ Then:
\begin{itemize}
  \item[(i)] $A\land x\land A=A\land y\land A$ if and only if $A_L\land x_L\land A_L=A_L\land y_L\land A_L$ and $A_R\land x_R\land A_R=A_R\land y_R\land A_R$;
  \item[(ii)] $B\lor u\lor B=B\lor v\lor B$ if and only if $B_L\lor u_L\lor B_L=B_L\lor v_L\lor B_L$ and $B_R\lor u_R\lor B_R=B_R\lor v_R\lor B_R$.
\end{itemize}
\end{lemma}

Propositions \ref{coset-full-right} and Corollary \ref{fullright} that follow describe the relation between the left [right] cosets and the full cosets of a skew lattice.

\begin{proposition}\label{coset-full-right}
Let $S$ be a skew lattice and let $A>B$ be $\Drel$-classes as above. Given $x,y\in B$ and $u,v\in A$ the following hold:
\begin{itemize}
   \item[(i)] $x\land A = y\land A$ if and only if $x_L=y_L$ and $x_R\land A_R=y_R\land A_R$;
     \item[(ii)] $A\land x = A\land y$ if and only if $x_R=y_R$ and $A_L\land x_L=A_L\land y_L$;
     \item[(iii)] $B\lor u = B\lor v$ if and only if $u_L=v_L$ and $B_R\lor u_R=B_R\lor v_R$;
     \item[(iv)] $u\lor B = v\lor B$ if and only if $u_R=v_R$ and $u_L\lor B_L=v_L\lor B_L$.
\end{itemize}   
\end{proposition}

\begin{proof}
We shall only prove (i) as the proofs of (ii)--(iv) are similar. Given $a\in A$ the following sequence of equivalences hold:
\begin{multline*}
  x\land A=y\land A \Leftrightarrow x\land a=y\land a \Leftrightarrow (x_L\land a_L, x_R\land a_R) =(y_l\land a_L, y_R\land a_R) \Leftrightarrow \\
   (x_L, x_R\land a_R) =(y_l, y_R\land a_R) \Leftrightarrow (x_L=y_L)\, \& \, (x_R\land A_R= y_R\land A_R).
\end{multline*}
Notice that we used the fact that $A_L$ is left-handed and Proposition \ref{lem_reg}.
\end{proof}

Let $R$ be a ring and $E(R)$ the set of all idempotent elements in $R$. 
Set $x\land y=xy$ and $x\lor y=x\circ y=x+y-xy$. If a set $S\subseteq E(R)$ is closed under both $\cdot $ and $\circ $ then $(S;\cdot , \circ )$ is a skew lattice.
By a \emph{skew lattice in a ring}\index{skew lattice in a ring} $R$ we mean a set $S\subseteq E(R)$ that is closed under the multiplication $\cdot$ and the operation $\nabla$ defined by: 
\[
x\nabla y=(x\circ y)^2=x+y+yx-xyx-yxy,
\]
and forms a skew lattice for the two operations. 
In particular, in addition to $S$ being closed under the two operations we need to ensure that $\nabla$ is associative on $S$. 
Given a multiplicative band $ B$ in a ring $R$ the relation between $\circ$ and $\nabla$ is given by $e\nabla f=(e\circ f)^2$ for all $e,f\in B$. 
In the case of right-handed skew lattices the nabla operation reduces to the circle operation; the same is true for left-handed skew lattices.   
%
%

%

Based on the standard form for pure bands in matrix rings that was developed by Fillmore et al. in \cite{Fi94} and \cite{Fi99}, Cvetko-Vah described in \cite{Ka07} the standard form for right-handed skew lattices in rings of matrices.
Let $F$ be a field of characteristic different than 2, $M_{n}(F)$ the ring of all $n\times n$ matrices over $F$ and $S\subseteq M_{n}(F)$ a primitive skew lattice with two comparable $\DD$-classes $A>B$. Then a basis for $F^{n}$ exists such that in this basis both $A$ and $B$ contain a diagonal matrix, the two diagonal matrices in $S$ form a lattice, and given any matrices $a\in A$ and $b\in B$, $a$ and $b$ have block forms:
\[
\begin{array}{cc}
a= \begin{bmatrix}
I & 0 &  a_{13} \\
0 & I &  a_{23} \\
a_{31} & a_{32} & a_{31}a_{13}+a_{32}a_{23} 
\end{bmatrix}\text{  and  }
&
b= \begin{bmatrix}
I & b_{12} & b_{13} \\
b_{21} & b_{21}b_{12} & b_{21}b_{13} \\
b_{31} & b_{31}b_{12} & b_{31}b_{13} 
\end{bmatrix}
\end{array}
\]
Denote the diagonal matrices in $A$ and $B$ by $a_0$ and $b_0$, respectively. If $S$ is right-handed then $aa_0=a_0$ and $b b_0=b_0$ which implies $a_{31}=a_{32}=0=b_{21}=b_{31}$. Thus $a$ and $b$ have block forms:
 \[
 \begin{array}{cc}
a= \begin{bmatrix}
I & 0 &  a_{13} \\
0 & I &  a_{23} \\
0 & 0 & 0 
\end{bmatrix}\text{  and  }
&
b= \begin{bmatrix}
I & b_{12} & b_{13} \\
0 & 0 & 0 \\
0 & 0 & 0
\end{bmatrix},
\end{array}
\]
 $bA=\set{ba:a\in A}$ is the coset of $A$ in $B$ that contains $b$, and $B\circ a=\set{b+a-ba:b\in B}$ is the coset of $B$ in $A$ that contains $a$.

On the other hand, if $S$ is left-handed then $a_{13}=a_{23}=0=b_{12}=b_{13}$ and thus
$a$ and $b$ have block forms:
\begin{center}
 $\begin{array}{cc}
a= \begin{bmatrix}
I & 0 &  0 \\
0 & I &  0 \\
a_{31} & a_{32} & 0 
\end{bmatrix}\text{  and  }
&
b= \begin{bmatrix}
I & 0 & 0 \\
b_{21} & 0 & 0 \\
b_{31} & 0 & 0
\end{bmatrix}
\end{array}$
\end{center}

Let $S$ be right-handed. Given matrices $a,a'\in A$ we obtain:
\[
B\circ a =B\circ a' \Leftrightarrow b_0\circ a = b_0\circ a' \Leftrightarrow 
a_{23}=a'_{23} ,
\]
and given $b,b'\in B$ we obtain:
\[
bA=b'A \Leftrightarrow ba_0=b'a_0 \Leftrightarrow b_{12}=b'_{12}.
\]
Similarly, if $S$ is left-handed we obtain: 
\[
a\circ B=a'\circ B \text{  iff  } a_{32}=a'_{32} \text{  and  } Ab=Ab' \text{  iff  } b_{21}=b'_{21}.
\] 
Let $S\subseteq M_{n}(F)$ be a primitive skew lattice with comparable $\DD$-classes $A>B$.
Let $S_{R}$ be the set of all upper triangular matrices of the form $a_R\in A_R$ or $b_R\in B_R$; these matrices have block forms:
 \[
 \begin{array}{cc}
a_{R}= \begin{bmatrix}
I & 0 &  a_{13} \\
0 & I &  a_{23} \\
0 & 0 & 0 
\end{bmatrix}\text{  and  }
&
b_{R}= \begin{bmatrix}
I & b_{12} & b_{13} \\
0 & 0 & 0 \\
0 & 0 & 0
\end{bmatrix}.
\end{array}
\]
 Similarly, let $S_{L}$ be the set of lower triangular matrices of the form $a_L\in A_L$ or $b_L\in B_L$; these matrices have block forms:
\[
\begin{array}{cc}
a_{L}= \begin{bmatrix}
I & 0 &  0 \\
0 & I &  0 \\
a_{31} & a_{32} & 0 
\end{bmatrix}\text{  and  }
&
b_{L}= \begin{bmatrix}
I & 0 & 0 \\
b_{21} & 0 & 0 \\
b_{31} & 0 & 0
\end{bmatrix}.
\end{array}
\]
 Then  $a= a_{L}\cdot a_{R}$ and 
$b=b_{L}\cdot b_{R}$.

Let $S$ be a skew lattice in $M_n(F)$, $A>B$ comparable $\DD$-classes in $S$, $x,y\in B$ and $u,v \in A$.
Then by Lemma \ref{fullcosets-left-right}:
\begin{itemize}
\item[(i)] $A x A=A y A$ if and only if $x_{21}=y_{21}$ and $x_{12}=y_{12}$, and
\item[(ii)] $B\nabla u\nabla B=B\nabla v\nabla B$ if and only if $u_{32}=v_{32}$ and $u_{23}=v_{23}$.
\end{itemize}
Similarly, Proposition \ref{coset-full-right} implies: 
 \begin{itemize}
\item[(i)] $xA=yA$ if and only if $x_{21}=y_{21}$, $x_{31}=y_{31}$  and  $x_{12}=y_{12}$;
 \item[(ii)] $Ax=Ay$ if and only if $x_{21}=y_{21}$, $x_{12}=y_{12}$ and $x_{13}=y_{13}$;
  \item[(iii)] $B\lor u=B\lor v$ if and only if $u_{31}=v_{31}$, $u_{32}=v_{32}$ and $u_{23}=v_{23}$.
  \item[(iv)] $u\lor B=v\lor B$ if and only if $u_{32}=v_{32}$, $u_{13}=v_{13}$ and $u_{23}=v_{23}$.
\end{itemize}
From the above equivalences we can thus observe that being in the same flat coset is a relation determined by the equalities $x_{31}=y_{31}$ and $x_{13}=y_{13}$ in the lower coset case, or $u_{32}=v_{32}$ and $u_{23}=v_{23}$ in the upper coset case.
This gives us a description extending the one given in \cite{Ka07}.

\end{document}